\documentclass[dvips]{imsart}

\RequirePackage[OT1]{fontenc} \RequirePackage{amsthm,amsmath}
\RequirePackage[numbers]{natbib}
\RequirePackage[colorlinks,citecolor=blue,urlcolor=blue]{hyperref}
\RequirePackage{hypernat}
\usepackage[T1]{fontenc}
\usepackage{amssymb}
\usepackage{dsfont}
\usepackage{graphicx}
\usepackage{amsmath}
\usepackage{amssymb}
\usepackage{amsmath}
\usepackage{amsfonts}
\usepackage[american]{babel}
\usepackage{graphicx}
\usepackage{flafter}
\usepackage[section]{placeins}

\startlocaldefs
\def\P{{\mathbb P}}
\def\E{{\mathbb E }}

\def\Bbb E{\mathbb{E}}
\def\Bbb R{\mathbb{R}}
\newtheorem{assumption}{Assumption}
\newtheorem{proposition}{Proposition}
\newtheorem{lemma}{Lemma}
\newtheorem{theorem}{Theorem}
\newtheorem{corollary}{Corollary}
\makeatletter \@addtoreset{equation}{section}

\makeatother

\font\tencmmib=cmmib10 \skewchar\tencmmib '60
\newfam\cmmibfam
\textfont\cmmibfam=\tencmmib

\font\tenmsb=msbm10 
\def\Bbb#1{\hbox{\tenmsb#1}}

\def\lessim{\ \lower4pt\hbox{$
\buildrel{\displaystyle <}\over\sim$}\ }
\def\gessim{\ \lower4pt\hbox{$\buildrel{\displaystyle >}
\over\sim$}\ }

\def\AA{{\cal A} }

\def\go0{\to 0}

\def\leftitem#1{\item{\hbox to\parindent{\enspace#1\hfill}}}

\def\sg{\sigma}

\def\sg2{\sigma^2}

\def\__{_{\infty}}
\def\E{\mathbb E}
\def\R{\mathbb R}
\endlocaldefs

\begin{document}
\begin{frontmatter}
\title{Optimal spectral norm rates for noisy low-rank matrix completion} \runtitle{Optimal spectral norm rates for noisy low-rank matrix completion}

\begin{aug}
\author{\fnms{Karim} \snm{Lounici}\thanksref{t1}\ead[label=e1]{klounici@math.gatech.edu}}

\thankstext{t1}{supported in part by NSF grant DMS 1106644 and Simons
foundation grant 209842.} \runauthor{K.
Lounici}

\affiliation{Georgia Institute of Technology\thanksmark{m1}}

\address{School of Mathematics\\
Georgia Institute of Technology\\
Atlanta, GA 30332-0160\\
\printead{e1}\\
\phantom{E-mail:\ }}

\end{aug}

\begin{abstract}
In this paper we consider the trace regression model where $n$
entries or linear combinations of entries of an unknown $m_1\times
m_2$ matrix $A_0$ corrupted by noise are observed. We establish for
the nuclear-norm penalized estimator of $A_0$ introduced in
\cite{KLT} a general sharp oracle inequality with the spectral norm
for arbitrary values of $n,m_1,m_2$ under an incoherence condition
on the sampling distribution $\Pi$ of the observed entries. Then, we
apply this method to the matrix completion problem. In this case, we
prove that it satisfies an optimal oracle inequality for the
spectral norm, thus improving upon the only existing result
\cite{KLT} concerning the spectral norm, which assumes that the
sampling distribution is uniform. Note that our result is valid, in
particular, in the high-dimensional setting $m_1m_2\gg n$. Finally
we show that the obtained rate is optimal up to logarithmic factors
in a minimax sense.
\end{abstract}

\begin{keyword}[class=AMS]
\kwd[Primary ]{62J99,62H12} \kwd[; secondary ]{60B20, 60G15}
\end{keyword}

\begin{keyword}
\kwd{matrix completion} \kwd{low-rank matrix estimation}
\kwd{spectral norm} \kwd{optimal rate of
convergence}\kwd{noncommutative Bernstein inequality} \kwd{Lasso}
\end{keyword}

\end{frontmatter}

\section{Introduction}

Consider $n$ independent observations $(X_i,Y_i), i=1,\dots,n$,
satisfying the trace regression model:
\begin{equation}\label{tracereg1}
Y_i = \mathrm{tr}(X_i^\top A_0) + \xi_i, \quad i=1,\dots,n,
\end{equation}
where $X_i$ are random matrices with dimensions $m_1\times m_2$,
$Y_i$ are random variables in $\R$, $A_0 \in \mathbb R^{m_1\times
m_2}$ is an unknown matrix, $\xi,\xi_i,i=1,\ldots, n$ are i.i.d.
zero mean random variables with $\sigma_\xi^2 = \mathbb E \xi^2
<\infty$ and $\mathrm{tr}(B)$ denotes the trace of matrix $B$. We
consider the problem of estimation of $A_0$ based on the
observations $(X_i,Y_i),\, i=1,\dots,n$.

For any matrices $A,B\in \R^{m_1\times m_2}$, we define the scalar
products
$$
\langle A,B \rangle = \mathrm{tr}(A^\top B),
$$
and
$$\langle A,B \rangle_{L_2(\Pi)} = \frac{1}{n}\sum_{i=1}^n
{\mathbb E}\big(\langle A,X_i\rangle \langle B,X_i\rangle\big)\,.
$$
Here $\Pi= \frac{1}{n}\sum_{i=1}^n \Pi_i$, where $\Pi_i$ denotes the
distribution of $X_i$. The corresponding norm $\|A \|_{L_2(\Pi)}$ is
given by $$\|A \|_{L_2(\Pi)}^2 = \frac{1}{n}\sum_{i=1}^n {\mathbb
E}\big(\langle A,X_i\rangle^2\big)\,.
$$

\textbf{Example 1: Matrix completion.} Let the design matrices $X_i$
be i.i.d. with distribution $\Pi$ on the set
\begin{equation}\label{matrixcompletion}
\mathcal{X} = \left\{ e_j(m_1) e_k^\top (m_2), 1\leq j \leq m_1,\,
1\leq k \leq m_2 \right\},
\end{equation}
where $e_k(m)$ are the canonical basis vectors in $\R^{m}$. The set
$\mathcal{X}$ forms an orthonormal basis in the space of $m_1\times
m_2$ matrices that will be called the matrix completion basis. Let
also $n< m_1m_2$. Then the problem of estimation of $A_0$ coincides
with the problem of matrix completion with random sampling
distribution $\Pi$. Existing results typically assume that $\Pi$ is
the uniform distribution on $\mathcal X$. See, for instance,
\cite{Gross-2,Recht} for the non-noisy case ($\xi_i=0,i=1,\ldots,n$)
and \cite{KLT} for the noisy case and the references cited therein.
In several applications, like the Netflix problem, the distribution
$\Pi$ is not necessarily uniform on $\mathcal X$. We will show that
optimal estimation of $A_0$ is possible in this context under a
weaker set of conditions as compared to those used in
\cite{Candes_Tao,Gross-2,Recht}. One can also consider other matrix
measurement models. For instance, \cite{Keshavan} considers sampling
without replacement in the set $\mathcal X$ defined in
(\ref{matrixcompletion}) and \cite{Kol10} investigates several
orthonormal families in the context of Quantum tomography.

\medskip

\textbf{Example 2. Column masks.} Let the design matrices $X_i$ be
independent matrices, which have only one nonzero column. The trace
regression model can be then reformulated as a longitudinal
regression model, with different distributions of $X_i$
corresponding to different tasks; see \cite{argyr08,LPTV,Rohde} for
more details and the references cited therein.

\medskip

\textbf{Example 3. "Complete" subgaussian design.} Let the design
matrices $X_i$ are i.i.d. replications of a random matrix $X$ such
that $\langle A,X \rangle$ is a subgaussian random variable for any
$A\in \R^{m_1\times m_2}$. This approach originates from compressed
sensing, where typically the entries of $X$ are either i.i.d.
standard Gaussian or Rademacher random variables. The problem of
exact reconstruction of $A_0$ under such a design in the non-noisy
setting was studied in \cite{rfp07,Candes_Plan,nw09}, whereas
estimation of $A_0$ in the presence of noise is analyzed in
\cite{nw09,Rohde,Candes_Plan}, among which
\cite{Rohde,Candes_Plan,Kol10} treat the high-dimensional case
$m_1m_2>n$.

%
%

We consider the following procedure introduced recently in
\cite{KLT}
\begin{equation}\label{estimateurPiknown}
\hat A^{\lambda} = \mathrm{argmin}_{A\in \R^{m_1\times m_2}}
\left\lbrace \|A\|_{L_2(\Pi)}^2 - \left\langle
\frac{2}{n}\sum_{i=1}^n Y_iX_i,A \right\rangle + \lambda\|A\|_1
\right\rbrace,
\end{equation}
where $\lambda>0$ is the regularization parameter and $\|A\|_1$ is
the nuclear norm of $A$.

In the matrix completion problem, the sampling scheme is typically
assumed to be uniform on $\mathcal X$ and this assumption is crucial
to establish the exact recovery in the noiseless case or to derive
the optimal rates of estimation with the Frobenius norm in the
setting $n<m_1m_2$; see for instance \cite{Candes_Recht,Gross-2,KLT}
and the references cited therein. However, in several applications
such as the Netflix problem, the practitioner does not choose the
sampling scheme and the observed entries of $A_0$ are not guaranteed
to follow the uniform distribution. Therefore, the existing exact
recovery or estimation results do not cover this situation.

In this paper, we concentrate mainly on the matrix completion
problem. Our contributions are the following. First, we establish
for the estimator (\ref{estimateurPiknown}) the following result. If
$A_0$ is low rank, $\Pi$ satisfies an incoherence condition and, in
addition, some additional mild conditions are satisfied, then we
have for any $t>0$ with probability at least $1-e^{-t}$
\begin{equation}\label{oraclebut}
\| \hat A^\lambda - A_0 \|_\infty \leq C(\sigma\vee
a)\sqrt{m_1m_2}\sqrt{(m_1\vee m_2)\frac{t+\log(m_1+m_2)}{n}},
\end{equation}
where $C>0$ is a numerical constant, $a$ is a bound on the absolute
values of the entries of $A_0$ and $\|\cdot\|_\infty$ is the
spectral norm. Second, we show that the above rate is optimal (in
the minimax sense) up to logarithmic factors on a particular class
of low rank matrices.

Note that the existing estimation results concern usually the
Frobenius norm \cite{Candes_Plan,Keshavan,Kol10,negaban10}. The only
existing estimation result for the spectral norm is due to
\cite{KLT} which assumes that the entries are sampled uniformly at
random. In this case, the estimator (\ref{estimateurPiknown}) can be
computed directly by soft-thresholding of the singular values in the
SVD of $ \mathbb X = \frac{m_1m_2}{n}\sum_{i=1}^n Y_iX_i$ (see
Equation (3.2) in \cite{KLT}). Exploiting this explicit simple form,
\cite{KLT} established (\ref{oraclebut}) for the procedure
(\ref{estimateurPiknown}). This approach does not generalize to
other sampling distribution $\Pi$ since (\ref{estimateurPiknown})
does not admit an explicit form in general. In this paper, we
propose an alternative approach to derive for the estimator
(\ref{estimateurPiknown}) the oracle inequality (\ref{oraclebut}
when the sampling distribution $\Pi$ satisfies an incoherence
condition, which covers in particular the case of uniform sampling
$\Pi$ and also holds in more general situations.

Note finally that the results of this paper are obtained for general
settings of $n,m_1,m_2$. In particular they are valid in the
high-dimensional setting, which corresponds to $m_1m_2\gg n$, with
low rank matrices $A_0$.

In section \ref{S2:tools}, we recall some tools and definitions and
establish a preliminary result. In Section \ref{S3:generaloracle},
we establish a general oracle inequality for the spectral norm. In
Section \ref{S4:matrixcompletion}, we apply the general result of
the previous section to the matrix completion problem and establish
the optimality (up to logarithmic factors) of
(\ref{estimateurPiknown}). Finally, Section \ref{S6:appendix}
contains additional material and proofs.

\section{Tools and preliminary result}\label{S2:tools}

We recall first some basic facts about matrices. Let
$A\in\R^{m_1\times m_2}$ be a rectangular matrix, and let $r={\rm
rank} (A) \leq \min(m_1,m_2)$ denote its rank. The singular value
decomposition (SVD) of $A$ admits the form
$$A=\sum_{j=1}^r \sigma_j(A)u_j^{(A)} \otimes  v_j^{(A)},$$
with orthonormal vectors $u_1^{(A)},\dots, u_r^{(A)}\in {\mathbb
R}^{m_1}$, orthonormal vectors $v_1^{(A)},\ldots,v_r^{(A)}\in
{\mathbb R}^{m_2}$ and real numbers $\sigma_1(A) \ge\dots\ge
\sigma_r(A)>0$ (the singular values of $A$). The pair of linear
vector spaces $(S_1(A),S_2(A))$ where $S_1(A)$ is the linear span of
$\{u_1^{(A)},\dots, u_r^{(A)}\}$ and $S_2(A)$ is the linear span of
$\{v_1^{(A)},\dots, v_r^{(A)}\}$ will be called the \it support \rm
of $A$. We will denote by $S_j(A)^{\perp}$ the orthogonal
complements of $S_j(A)$, $j=1,2$, and by $P_{S}$ the orthogonal
projector onto the linear vector subspace $S$ of ${\mathbb
R}^{m_j}$, $j=1,2$. For any $A\in {\mathbb A}$ with support
$(S_1,S_2),$ we define
$$
{\cal P}_A (B):=B-P_{S_1^{\perp}}BP_{S_2^{\perp}},\ {\cal
P}_A^{\perp}(B):=P_{S_1^{\perp}}BP_{S_2^{\perp}}, \ B\in {\mathbb
R}^{m_1\times m_2}.
$$

The Schatten-$p$ (quasi-)norm $\| A \|_p$  of matrix $A$ is defined
by
$$
\| A \|_p = \left( \sum_{j=1}^{\min(m_1,m_2)}\sigma_j(A)^p
\right)^{1/p} \ \text{for} \ 0< p<\infty,\quad \text{and}\quad
\|A\|_{\infty}=\sigma_1(A).
$$

Recall the well-known {\it trace duality} property:
$$
\left| \mathrm{tr}(A^\top B) \right| \leq \|A\|_1 \|B\|_{\infty},
\quad \forall A,B\in \R^{m_1\times m_2}.
$$

We will also use the fact that the subdifferential of the convex
function $A\mapsto \|A\|_1$ is the following set of matrices:
\begin{equation}
\label{subdiff}
\partial \|A\|_1=\Bigl\{\sum_{j=1}^r u_j^{(A)} \otimes v_j^{(A)} +
P_{S_1(A)^{\perp}}W P_{S_2(A)^{\perp}}:\ \|W\|_\infty\leq 1 \Bigr\}
\end{equation}
(cf. \cite{watson}).

We will need the following quantities introduced in
\cite{Kolstflour}
$$
\kappa_r = \kappa_r(\Pi) := \inf\left\lbrace \|B_2\|_{L_2(\Pi)}: \; B\in
\R^{m_1\times m_2},\, \|B\|_2 = 1,\, \mathrm{rank}(B)\leq r
\right\rbrace
$$
and
$$
\kappa'_r = \kappa'_r(\Pi) := \sup\left\lbrace \|B_2\|_{L_2(\Pi)}: \; B\in
\R^{m_1\times m_2},\, \|B\|_2 = 1,\, \mathrm{rank}(B)\leq r
\right\rbrace.
$$
These quantities  $\kappa_r(\Pi)$ and $\kappa'_r(\Pi)$ measure the "distorsion"
on the set of low rank matrices between the geometries induced
respectively by the $L_2(\Pi)$ and Frobenius norms.

We introduce the following measure of coherence
\begin{align}
 \rho &= \rho(\Pi) :=
\sup \left\lbrace   \frac{\left|\langle A,B
\rangle_{L_2(\Pi)}\right|}{\|A\|_{1}\|B\|_{1}} \, : \, \forall A,B
\in \R^{m_1\times m_2}, \langle A,B \rangle = 0 \right\rbrace.
\end{align}
We can now state our incoherence condition.
\begin{assumption}\label{mutcoh}
Let $c_0\geq 0$, $\alpha >1$ and $r\geq 1$. We have
$$
\rho \leq \frac{\kappa_1^2}{(1+2c_0)\alpha  r},
$$
\end{assumption}
The quantity $\rho$ is the natural extension to the matrix case of
the incoherence measure introduced for the sparse vector case in
\cite{DET06} and further studied in \cite{B07,BTW07,L08} and the
references cited therein. Concerning the matrix completion problem,
\cite{Candes_Plan,Candes_Recht,Gross-2,KLT} study the case of
uniform at random sampling. Assumption \ref{mutcoh} is then
trivially satisfied with $\rho=0$, since we have $\langle
A,B\rangle_{L_2(\Pi)} = \frac{1}{m_1 m_2}\langle A,B\rangle$ for any
$A,B\in\R^{m_1\times m_2}$. Note also that
\cite{Candes_Plan,Candes_Recht,Gross-2} need in addition the
following condition in order to recover $A_0$ in the noiseless case
\begin{align*}
\label{incoh1}\max_{u\in \mathcal X}| \mathcal{P}_{S_1,S_2}(u) |_2^2
\leq \frac{2\nu r}{m_1 \wedge m_2},\quad \left|
\sum_{j=1}^{r}u_j^{(A_0)} \otimes v_j^{(A_0)}  \right|_{\infty} \leq
\frac{2\nu r}{(m_1\wedge m_2)^2},
\end{align*}
for some $\nu>0$ where $S_j=S_j(A_0)$, $j=1,2$ and $|\cdot|_2$,
$|\cdot|_\infty$ denote respectively the $l_2$ and $l_\infty$ vector
norms. Although called "incoherence condition" in \cite{Gross-2},
this condition is entirely different from Assumption \ref{mutcoh}
and we do not need it to establish our estimation result.

In \cite{KLT}, the authors establish an oracle inequality for the
$L_2(\Pi)$ norm under a condition akin to the restricted eigenvalue
condition in sparse vector estimation: $\mu_{c_0}(A_0)<\infty$ for
some $c_0 \geq 0$ where
$$
\mu_{c_0}(A_0):=\inf\Bigl\{\mu>0: \|{\cal P}_{A_0} (B)\|_2 \leq \mu
\|B\|_{L_2(\Pi)}, \,\forall B\in {\mathbb C}_{A_0,c_0}\Bigr\},
$$
and ${\mathbb C}_{A_0,c_0}$ is the following cone of matrices
$$
{\mathbb C}_{A_0,c_0}:=\Bigl\{B\in {\mathbb R}^{m_1\times m_2}:
\|{\cal P}_{A_0}^{\perp}(B)\|_1\leq c_0 \|{\cal
P}_{A_0}(B)\|_1\Bigr\}.
$$
Note that $\mu_{c_0}(A_0)$ is a nondecreasing function of $c_0$. We
establish in Proposition \ref{Mutcoh-RE} below that Assumption
\ref{mutcoh} implies
$\mu_{c_0}(A_0)<\frac{1}{\kappa_1}\sqrt{\frac{\alpha}{\alpha-1}}$ if
$\mathrm{rank}(A_0)\leq r$.

\begin{proposition}\label{Mutcoh-RE}
  Let Assumption \ref{mutcoh} be satisfied for some $c_0\geq 0$, $\alpha >1$ and $r\geq
  1$. Assume furthermore that $\kappa_1=\kappa_1(\Pi)>0$. Then, for any $A\in \R^{m_1\times m_2}$ with $\mathrm{rank}(A)\leq r$, we have
  $$
\mu_{c_0}(A) \leq
\frac{1}{\kappa_1}\sqrt{\frac{\alpha}{\alpha-1}}<\infty.
  $$
\end{proposition}
\begin{proof}
We have
\begin{eqnarray}
 \|\mathcal{P}_A(B)
+\mathcal{P}_{A}^{\perp}(B)\|_{L_2(\Pi)}^2 &=&
\|\mathcal{P}_A(B)\|_{L_2(\Pi)}^2 +
\|\mathcal{P}_A^{\perp}(B)\|_{L_2(\Pi)}^2 + 2 \langle
\mathcal{P}_A(B), \mathcal{P}_A^{\perp}(B)\rangle_{L_2(\Pi)}\notag\\
&\geq& \|\mathcal{P}_A(B)\|_{L_2(\Pi)}^2 +
\|\mathcal{P}_A^{\perp}(B)\|_{L_2(\Pi)}^2 - 2 \rho
\|\mathcal{P}_A(B)\|_1 \|\mathcal{P}_A^{\perp}(B)\|_1 \notag\\
&\geq& \|\mathcal{P}_A(B)\|_{L_2(\Pi)}^2  - 2 \rho c_0
\|\mathcal{P}_A(B)\|_1^2 \notag\\
&\geq& \|\mathcal{P}_A(B)\|_{L_2(\Pi)}^2  - 2 \rho c_0 r
\|\mathcal{P}_A(B)\|_2^2.\label{intermMutcoh1}
\end{eqnarray}

Next, we treat $\|\mathcal{P}_A(B)\|_{L_2(\Pi)}^2$. For the sake of
brevity, we set $r=\mathrm{rank}(\mathcal{P}_A(B))$ and, for any
$1\leq j \leq r$, $\sigma_j = \sigma_j(\mathcal P _{A}(B))$, $u_j =
u_j^{(\mathcal{P}_A(B))}$ and $v_j =v_j^{(\mathcal{P}_A(B))}$.
Recall that the SVD of $\mathcal{P}_A(B)$ is
$$
\mathcal{P}_A(B) = \sum_{j=1}^{r} \sigma_j u_j \otimes v_j.
$$
For any $B\in\R^{m_1\times m_2}$, we have
\begin{eqnarray}
\|\mathcal{P}_A(B)\|_{L_2(\Pi)}^2 &=& \|\sum_{j=1}^r \sigma_j u_j
\otimes v_j\|_{L_2(\Pi)}^2\notag\\
&=& \sum_{j=1}^r \sigma_j^2 \| u_j\otimes v_j \|_{L_2(\Pi)}^2 +
\sum_{j,k=1:j\neq k}^r \sigma_j \sigma_k \langle u_j \otimes v_j ,
u_k
\otimes v_k\rangle_{L_2(\Pi)} \notag\\
&\geq& \sum_{j=1}^r \sigma_j^2 \| u_j \otimes v_j \|_{L_2(\Pi)}^2 -
 \rho \sum_{j,k=1:j\neq k}^r \sigma_j \sigma_k \notag\\
 &\geq& \sum_{j=1}^r \sigma_j^2 \| u_j \otimes v_j \|_{L_2(\Pi)}^2 -
 \rho \left(\sum_{j=1}^r \sigma_j\right)^2 \notag\\
 &\geq& (\kappa_1^2 - \rho r)\sum_{j=1}^r \sigma_j^2 = (\kappa_1^2 - \rho r)\| \mathcal{P}_A(B)
 \|_2^2.\label{intermMutcoh2}
\end{eqnarray}
Combining (\ref{intermMutcoh1}) and \ref{intermMutcoh2} with
Assumption \ref{mutcoh} yields
\begin{align*}
\|B\|_{L_2(\Pi)}^2 & \geq \left( \kappa_1^2- \rho(1+ 2
c_0)r\right)\|\mathcal{P}_A(B)\|_2^2\notag\\
&\geq \frac{\kappa_1^2(\alpha -1)}{\alpha}\|\mathcal{P}_A(B)\|_2^2.
\end{align*}
Thus, we get the result.
\end{proof}
\section{General oracle inequalities for the spectral norm}\label{S3:generaloracle}

Define the random matrices
\begin{equation}\label{randM}
{\bf M}_1 = \frac{1}{n}\sum_{i=1}^n \xi_i X_i,\quad {\bf M}_2=
\frac{1}{n}\sum_{i=1}^n \langle A_0, X_i \rangle - \mathbb E(\langle
A_0, X_i \rangle).
\end{equation}

We can now state the main result, which holds for general settings
including in particular the three examples presented in the
introduction.
\begin{theorem}\label{thmainPiknown}
Let Assumption \ref{mutcoh} be satisfied with $c_0=5$ and
$\mathrm{rank}(A_0)\leq r$. Then, the estimator
(\ref{estimateurPiknown}) satisfies on the event $\lambda \geq
3\|{\bf M}_1 + {\bf M}_2\|_\infty$
\begin{equation}\label{mainresult}
\|\hat A^{\lambda} - A_0\|_{\infty} \leq \left(\frac{5}{6} +
\frac{6\sqrt{2}}{11(\alpha-1)}\right)\frac{\lambda}{\kappa_1^2}.
\end{equation}
\end{theorem}
In \cite{KLT}, the authors obtained an oracle inequality for the
Frobenius norm with an upper bound proportional to
$\mathrm{rank}(A_0)\lambda/\kappa_1^2$ (with our notations), which
trivially implies a suboptimal bound for the spectral norm since
$\|\cdot\|_\infty \leq \|\cdot\|_2$. Under Assumption \ref{mutcoh},
we obtain a bound (\ref{mainresult}) that does not depend on
$\mathrm{rank}(A_0)$. We will see in Section
\ref{S4:matrixcompletion} that this oracle inequality gives the
optimal rate for the spectral norm in the matrix completion problem.

\begin{proof}
Note first that a necessary condition of extremum in the
minimization problem (\ref{estimateurPiknown}) implies that there
exists $\hat V \in
\partial \|\hat A^{\lambda} \|_1$ such that, for all $A\in \R^{m_1\times m_2}$
\begin{equation*}
2\langle\hat A^\lambda ,\hat A^\lambda - A\rangle_{L_2(\Pi)} -
\left\langle \frac{2}{n}\sum_{i=1}^n Y_i X_i, \hat A^\lambda - A
\right\rangle + \lambda \langle  \hat V, \hat A^\lambda - A \rangle
= 0.
\end{equation*}
Set $\Delta = \hat A^{\lambda} - A_0$. It follows from the previous
display that, for any $U\in \R^{m_1\times m_2}$ with $\|U\|_1=1$,
\begin{equation*}
\left|\langle \Delta, U \rangle_{L_2(\Pi)}\right| \leq  \|{\bf M}_1
+ {\bf M}_2\|_\infty + \frac{\lambda}{2}.
\end{equation*}
Thus we get, on the event $ \lambda \geq 3\|{\bf M}_1 + {\bf
M}_2\|_\infty $, for any $U\in \R^{m_1\times m_2}$ with $\|U\|_2=1$,
\begin{equation}\label{ThPiknown-interm1}
\left|\langle \Delta, U \rangle_{L_2(\Pi)}\right| \leq
\frac{5}{6}\lambda.
\end{equation}
Next, recall that the SVD of $\Delta = \hat A^{\lambda} - A_0$ is
$$
\Delta =\sum_{j=1}^{\hat{r}} \sigma_j(\Delta)u_j^{(\Delta)}\otimes
v_j^{(\Delta)},\quad \hat r = \mathrm{rank}(\Delta).
$$
Take $U=u^{(\Delta)}_1 \otimes v^{(\Delta)}_1$. Then, we have
\begin{align*}
  \langle \Delta,U \rangle_{L_2(\Pi)} &= \langle P_1(\Delta),U
  \rangle_{L_2(\Pi)} + \langle P_{1}^{\perp}(\Delta),U
  \rangle_{L_2(\Pi)},
\end{align*}
where $P_1$ and $P_1^{\perp}$ denote the orthogonal projections onto
$M_1 = \mathrm{l.s.}\left(u_1^{(\Delta)}\otimes
v_1^{(\Delta)}\right)$ and $M_1^\perp$ respectively. Combining the
previous display with Equation (\ref{ThPiknown-interm1}) and
Assumption \ref{mutcoh} gives
\begin{equation*}\label{ThPiknown-interm2}
|\langle P_1(\Delta), U  \rangle_{L_2(\Pi)}| \leq \frac{5}{6}\lambda
+ \rho \|U\|_1 \|P_{1}^{\perp}(\Delta)\|_1 \leq \frac{5}{6}\lambda +
\rho  \|P_{1}^{\perp}(\Delta)\|_1\leq \frac{5}{6}\lambda + \rho
\|\Delta\|_1.
\end{equation*}

Lemma \ref{leminterm} yields on the event $\lambda\geq 3 \|{\bf M}_1
+ {\bf M}_2\|_\infty$ that
$$
\| \mathcal{P}_{A_0}^{\perp}(\Delta) \|_1 \leq 5
\|\mathcal{P}_{A_0}(\Delta) \|_1,
$$
which implies that $\Delta = \hat A^{\lambda}-A_0 \in
\mathbb{C}_{A_0,5}$. Combining the last two displays, we get on the
event $\lambda\geq 3 \|{\bf M}_1 + {\bf M}_2\|_\infty$
\begin{align*}
|\langle P_1(\Delta), U  \rangle_{L_2(\Pi)}| &\leq
\frac{5}{6}\lambda +
6\sqrt{2\mathrm{rank}(A_0)}\rho\|\mathcal{P}_{A_0}(\Delta)\|_2\\
&\leq \frac{5}{6}\lambda +
6\sqrt{2\mathrm{rank}(A_0)}\rho\mu_5(A_0)\|\Delta\|_{L_2(\Pi)}.
\end{align*}

Theorem 2 in \cite{KLT} with $A=A_0$ gives on the event $\lambda
\geq 3\|{\bf M}_1 + {\bf M}_2\|_\infty$
$$
\|\Delta\|_{L_2(\Pi)} \leq
\lambda\mu_5(A_0)\sqrt{\mathrm{rank}(A_0)}.
$$
Combining the last two displays, we get on the event $\lambda\geq 3
\|{\bf M}_1 + {\bf M}_2\|_\infty$
\begin{align*}
 |\langle P_1(\Delta), U \rangle_{L_2(\Pi)}|
&\leq \frac{5}{6}\lambda +  6\sqrt{2}\mathrm{rank}(A_0)\rho\mu_5(A_0)^2 \lambda\\
&\leq \left(\frac{5}{6} +
\frac{6\sqrt{2}}{11(\alpha-1)}\right)\lambda,
\end{align*}
where we have used Assumption \ref{mutcoh} and Proposition
\ref{Mutcoh-RE} in the second line.

Next, note that
\begin{align*}
\langle P_1(\Delta),U
  \rangle_{L_2(\Pi)} &= \sigma_1(\Delta) \|U\|^2_{L_2(\Pi)}\geq \sigma_1(\Delta) \kappa_1^2 \|U\|_2^2 =\sigma_1(\Delta)
  \kappa_1^2.
\end{align*}
Finally, combining the last two displays, we get the
result.\end{proof}

\section{Matrix completion upper bounds with the spectral norm}\label{S4:matrixcompletion}

In this section, we apply the general results of the previous
section to the matrix completion problem with i.i.d. sub-exponential
noise variables.
\begin{assumption}\label{sup-exp assum}
There exist constants $\sigma,c_1>0$, $\beta\ge 1$ and $\tilde c$
such that
\begin{equation}\label{subexp}
\max_{i=1,\dots,n}\E \exp
\left(\frac{|\xi_i|^\beta}{\sigma^\beta}\right) < \tilde c, \quad
\mathbb{E}\xi_i^2 \geq \bar c \sigma^2,\,\forall 1\leq i \leq n.
\end{equation}
\end{assumption}
We need the following additional condition on $\kappa_1$ and
$\kappa_1'$.
\begin{assumption}\label{assumkappa}
  There exist constants $0<c_1\leq c_1' <\infty$ such that
\begin{equation}
\sqrt{\frac{ c_1}{m_1 m_2}}\leq\kappa_1 \leq  \kappa_1' \leq
\sqrt{\frac{ c_1'}{m_1 m_2}}.
\end{equation}
\end{assumption}
This assumption imposes that the probability to observe any entry is
not too small or too large. It guarantees that any low-rank matrix
can be estimated with optimal spectral norm rate (up to logarithmic
factors). Indeed, when Assumption \ref{assumkappa} is satisfied, we
can establish that the stochastic errors $\|{\bf M}_1\|_\infty$ and
$\|{\bf M}_2\|_\infty$ are small enough with probability close to
$1$.

Set $m=m_1+m_2$ and $M = m_1\vee m_2$. Denote the entries of $A_0$
by $a_0(i,j),\, 1\leq i \leq m_1,\,1\leq j \leq m_2$. We can now
state our main results concerning matrix completion.

\begin{theorem}\label{thmatrixcompletionPiknown} Let $X_i$ be i.i.d. with distribution $\Pi$ on $\mathcal X$ defined in
(\ref{matrixcompletion}). Let Assumption \ref{mutcoh} be satisfied
with $c_0=5$ and $\mathrm{rank}(A_0)\leq r$. Let Assumptions
\ref{sup-exp assum} and \ref{assumkappa} with, in addition, $2c_1'
\leq M c_1$. Assume that $\max_{i,j}|a_0(i,j)|\leq a$ for some
constant $a$. For $t>0,$ consider the regularization parameter
$\lambda$ satisfying
\begin{equation}\label{eq:th:completion_gaussian}
\lambda \geq C(\sigma \vee a)\max\left\{\
\sqrt{\frac{t+\log(m)}{(m_1\wedge m_2)n}}\, , \
\frac{(t+\log(m))\log^{1/\beta}(m_1\wedge m_2)}{n} \right\},
\end{equation}
where $C>0$ is a large enough constant that can depend only on
$\alpha,\beta,\tilde c,\bar c, c_1, c_1'$. Then, the estimator
(\ref{estimateurPiknown}) satisfies, with probability at least
$1-e^{-t}$
\begin{equation}\label{supnormPiknown}
\|\hat A^{\lambda} - A_0\|_{\infty} \leq C'(\sigma \vee
a)m_1m_2\max\left\{ \sqrt{\frac{t+\log(m)}{(m_1\wedge m_2)n}}\, , \
\frac{(t+\log(m))\log^{1/\beta}(m_1\wedge m_2)}{n} \right\},
\end{equation}
where $C'>0$ can depend only on $\alpha,\beta,\tilde c,\bar c, c_1,
c_1'$.
\end{theorem}

Note that the technical condition $2c_1' \leq M c_1.$ is mild when
$M\geq 2$ is large. Note also that when the noise variables are
bounded, then this technical condition is no longer needed since we
can apply Proposition \ref{prop:Bernstein_bounded} instead of
Proposition \ref{prop:Bernstein_unbounded} in Section
\ref{S6:appendix} to control $\|{\bf M}_1\|_\infty$.

\begin{proof}
This proof consists in applying Theorem \ref{thmainPiknown} with a
sufficiently large $\lambda$ such that the condition $\lambda \geq
3\|{\bf M}_1 + {\bf M}_2\|_\infty$ holds with probability close to
$1$. To this end, we need to control the stochastic errors $\|{\bf
M}_1\|_\infty$ and $\|{\bf M}_2\|_\infty$; see Lemmas \ref{lem:tau1}
and \ref{lem:tau2} in Section \ref{S6:appendix} below. Next a simple
union bound argument gives for any $\lambda$ satisfying
(\ref{eq:th:completion_gaussian}) that (\ref{supnormPiknown}) holds
with probability at least $1-3e^{-t}$, which can then be rewritten
as $1-e^{-t}$ with a proper adjustment of the constants.
\end{proof}

Note that the natural choice of $t$ is of the order $\log (m)$. In
addition, if $n>M\log^{1+2/\beta}(m)$, then we choose $\lambda$ of
the form
\begin{equation}\label{optimal lambda}
\lambda = C(\sigma \vee a)\sqrt{\frac{\log (m)}{(m_1\wedge m_2)n}},
\end{equation}
where $C>0$ is a large enough constant that can depend only on
$\alpha,\beta,\tilde c,\bar c, c_1, c_1'$. We immediately obtain the
following corollary of Theorem \ref{thmatrixcompletionPiknown}

\begin{corollary}\label{cor:completionPiknown} Let the assumptions
of Theorem \ref{thmatrixcompletionPiknown} be satisfied with
$\lambda$ as in (\ref{optimal lambda}) and a large enough constant
$C>0$ that can depend only on $\alpha,\beta,\tilde c,\bar c, c_1,
c_1'$, $n>(m_1 \vee m_2) \log^{1+2/\beta}(m)$.

Then, the estimator (\ref{estimateurPiknown}) satisfies, with
probability at least $1-1/m$,
\begin{equation}
\label{first2_1} \|{\hat A}^\lambda-A_0\|_\infty \leq C' (\sigma\vee
a)\sqrt{m_1m_2}\sqrt{\frac{M\log m}{n}},
\end{equation}
where $C'>0$ can depend only on $\alpha,\beta,\tilde c,\bar c, c_1,
c_1'$.
\end{corollary}

We prove now that the above result is optimal up to logarithmic
factors by establishing a minimax lower bound. We will denote by
$\inf_{\hat{A}}$ the infimum over all estimators $\hat{A}$ with
values in $\R^{m_1\times m_2}$. For any integer $r\le \min(m_1,m_2)$
and any $a>0$ we consider the class of matrices
\begin{equation*}\label{Alb}
{\cal A}(r,a)= \big\{A_0\in\,\R^{m_1\times m_2}:\,
\mathrm{rank}(A_0)\leq r,\, \max_{i,j}|a_0(i,j)|\,\leq\, a\big\}\,.
\end{equation*}
For any $A\in\R^{m_1\times m_2}$, let $\P_A$ denote the probability
distribution of the observations $(X_1,Y_1,\dots,X_n,Y_n)$ with $\E
(Y_i|X_i)=\langle A,X_i\rangle$.

\begin{theorem}\label{th:lower_completion} Fix $a>0$ and an integer
$r$ such that $1\leq  r\leq m_1 \wedge m_2$, $M r\leq n$. Let the
matrices $X_i$ be i.i.d. with distribution $\Pi$ on $\mathcal X$
satisfying Assumption \ref{assumkappa}. Let the variables $\xi_i$ be
independent Gaussian ${\cal N}(0,\sigma^2)$, $\sigma^2>0$, for
$i=1,\dots,n$. Then there exist absolute constants $\beta\in(0,1)$
and $c>0$, such that
\begin{equation}\label{eq:lower2}
\inf_{\hat{A}}\ \sup_{\substack{A_0\in\,{\cal A}(r,a) }}
\P_{A_0}\bigg(\|\hat{A}-A_0\|_{\infty}> c(\sigma\wedge a)
\sqrt{m_1m_2}\sqrt{\frac{Mr}{n}} \bigg)\ \geq\ \beta.
\end{equation}
\end{theorem}
The proof of this result can be found in Section 6 below.

Comparing Theorem~\ref{th:lower_completion} with
Corollary~\ref{cor:completionPiknown} we see that, in the case of
Gaussian errors $\xi_i$, the rate of convergence of $\hat A^\lambda$
is optimal (up to a logarithmic factor) in a minimax sense on the
class of matrices ${\cal A}(r,a)$.

\section{Proofs}\label{S6:appendix}

\subsection{An intermediate result}

We need the following lemma to prove Theorem \ref{thmainPiknown}.
\begin{lemma}\label{leminterm}
The estimator (\ref{estimateurPiknown}) satisfies, on the event
$\lambda \geq 3 (\|{\bf M}_1+{\bf M}_2\|_\infty )$
$$
\| \mathcal{P}_{A_0}^{\perp}(\hat A^{\lambda}-A_0) \|_1 \leq 5
\|\mathcal{P}_{A_0}(\hat A^\lambda-A_0) \|_1.
$$
\end{lemma}
Note that this result is an intermediate result in the proof of
Theorem 2 in \cite{KLT}. For the sake of completeness, we provide
here a proof of this result.

\begin{proof}
Note that a necessary condition of extremum in the minimization
problem (\ref{estimateurPiknown}) implies that there exists $\hat V
\in
\partial \|\hat A^{\lambda} \|_1$ such that, for all $A\in \R^{m_1\times m_2}$
\begin{equation*}
2\langle\hat A^\lambda ,\hat A^\lambda - A\rangle_{L_2(\Pi)} -
\left\langle \frac{2}{n}\sum_{i=1}^n Y_i X_i, \hat A^\lambda - A
\right\rangle + \lambda \langle  \hat V, \hat A^\lambda - A \rangle
= 0.
\end{equation*}
Set ${\bf M} = {\bf M}_1+ {\bf M}_2$. It follows from the previous
display that
\begin{equation*}
2\|\hat A^\lambda - A_0\|_{L_2(\Pi)}^2 + \lambda \left\langle \hat V
- V, \hat A^\lambda - A_0 \right\rangle = -\lambda \langle V, \hat
A^\lambda - A_0 \rangle + 2\langle {\bf M}, \hat A^\lambda - A_0
\rangle,
\end{equation*}
for an arbitrary $V\in \partial \|A_0\|_1$. For the sake of brevity,
we set $A_0 = \sum_{j=1}^r \sigma_j u_j \otimes v_j$ where
$r=\mathrm{rank}(A_0)$, $u_j=u_j^{(A_0)}$, $v_j=v_j^{(A_0)}$ and
$S_j=S_j(A_0)$, $j=1,2$. Then, $V$ admits the following
representation
$$
V= \sum_{j=1}^r u_j\otimes v_j + P_{S_1}^{\perp}WP_{S_2}^{\perp},
$$
where $W$ is an arbitrary matrix with $\|W\|_\infty \leq 1$. By
monotonicity of the subdifferential of convex functions,
$\left\langle \hat V - V, \hat A^\lambda - A_0 \right\rangle\geq 0$.
Therefore, we get
$$
\lambda \langle P_{S_1}^{\perp}WP_{S_2}^{\perp}, \hat A^\lambda -
A_0\rangle \leq -\lambda \left\langle \sum_{j=1}^r u_j\otimes v_j,
\hat A^\lambda - A_0  \right\rangle + 2\langle {\bf M}, \hat
A^\lambda - A_0 \rangle.
$$
Set $\Delta = \hat A^{\lambda} - A_0$. The trace duality guarantees
the existence of a matrix $W$ with $\|W\|_\infty$ such that
$$
\langle P_{S_1}^{\perp}WP_{S_2}^{\perp}, \Delta\rangle =\langle W,
P_{S_1}^{\perp}\Delta\lambda P_{S_2}^{\perp}\rangle =
\|P_{S_1}^{\perp}\Delta P_{S_2}^{\perp}\|_1.
$$
The trace duality again implies that
$$
\left| \left\langle \sum_{j=1}^r u_j\otimes v_j, \Delta\right
\rangle \right|\leq \|P_{S_1} \Delta P_{S_2}\|_1.
$$
Combining the last three displays, we get, on the event $\lambda
\geq 3 (\|{\bf M}_1+{\bf M}_2\|_\infty )$
\begin{align*}
\|P_{S_1}^{\perp}\Delta P_{S_2}^{\perp} \|_1 &\leq \|P_{S_1}\Delta
P_{S_2} \|_1 + \frac{2}{3}\|\Delta\|_1\\
&\leq \frac{5}{3}\|P_{S_1}\Delta P_{S_2} \|_1 +
\frac{2}{3}\|P_{S_1}^{\perp}\Delta P_{S_2}^{\perp}\|_1.
\end{align*}
Thus we get the result.
\end{proof}

\subsection{Control of the stochastic errors}
The following proposition is an immediate consequence of the matrix
version of Bernstein's inequality (Corollary 9.1 in \cite{tropp10}).
For the sake of brevity, we write $\|\cdot\|_\infty = \|\cdot\|$.

\begin{proposition}\label{prop:Bernstein_bounded}
Let $Z_1,\ldots,Z_n$ be independent random matrices with dimensions
$m_1\times m_2$ that satisfy $\E(Z_i) = 0$ and $\|Z_i\|\leq U$
almost surely for some constant $U$ and all $i=1,\dots,n$. Define
$$
\sigma_Z = \max\Bigg\{\,\Big\|\frac1{n}\sum_{i=1}^n\E
(Z_iZ_i^\top)\Big\|^{1/2},\, \,\Big\|\frac1{n}\sum_{i=1}^n\E
(Z_i^\top Z_i)\Big\|^{1/2}\Bigg\} .
$$
Then, for all $t>0,$ with probability at least $1-e^{-t}$ we have
$$
\left\| \frac{Z_1 + \cdots + Z_n}{n}  \right\| \leq 2\max\left\{
\sigma_Z\sqrt{\frac{t + \log (m)}{n}}\,, \ \ U \frac{t + \log
(m)}{n} \right\}\,,
$$
where $m=m_1+m_2$.
\end{proposition}
Furthermore, it is possible to replace the $L_\infty$-bound $U$ on
$\|Z\|$ in the above inequality by bounds on the weaker
$\psi_\beta$-norms of $\|Z\|$ defined by
$$
U_Z^{(\beta)} = \inf \Big\{u>0:\,  \E\exp(\|Z\|^\beta/u^\beta) \le
2\Big\}, \quad \beta\geq 1.
$$

\begin{proposition}\label{prop:Bernstein_unbounded}
Let $Z,Z_1,\ldots,Z_n$ be i.i.d. random matrices with dimensions
$m_1\times m_2$ that satisfy $\E(Z) = 0$. Suppose that
$U_Z^{(\beta)} <\infty$ for some $\beta\geq 1$. Then there exists a
constant $C>0$ such that, for all $t>0$, with probability at least
$1-e^{-t}$
$$
\left\| \frac{Z_1 + \cdots + Z_n}{n}  \right\| \leq C\max\left\{
\sigma_Z\sqrt{\frac{t + \log (m)}{n}}\,, \ \ U_Z^{(\beta)}\left(
\log \frac{U_Z^{(\beta)}}{\sigma_Z} \right)^{1/\beta} \frac{t + \log
(m)}{n} \right\},
$$
where $m=m_1+m_2$.
\end{proposition}
This is an easy consequence of Proposition 2 in \cite{Kol10}, which
provides an analogous result for Hermitian matrices $Z$. Its
extension to rectangular matrices stated in
Proposition~\ref{prop:Bernstein_unbounded} is straightforward via
the self-adjoint dilation; see, for example, the proof of
Corollary~9.1 in \cite{tropp10}.

\begin{lemma}\label{lem:tau1}
Let the noise variables $\xi_1,\ldots,\xi_n$ be i.i.d. and satisfy
Assumption \ref{sup-exp assum}. Let $X,X_1,\ldots,X_n$ be i.i.d.
with distribution $\Pi$ on $\mathcal X$ satisfying Assumption
\ref{assumkappa}. Then there exists an absolute constant $C>0$ that
can depend only on $\beta,\tilde c, \bar c, c_1, c_1'$ and such
that, for all $t>0,$ with probability at least $1-2e^{-t}$ we have
\begin{equation}\label{eq:lem:tau1}
\|{\bf M}_1\| \leq C\sigma\max\left\{
\sqrt{\frac{t+\log(m)}{(m_1\wedge m_2)n}}\, , \
\frac{(t+\log(m))\log^{1/\beta}(m_1\wedge m_2)}{n} \right\}.
\end{equation}
\end{lemma}

The proof of this lemma is essentially the same as that of Lemma 2
in \cite{KLT} up to some additional technicalities due to the fact
$\Pi$ is no longer assumed to be the uniform distribution on
$\mathcal X$. We set $\pi(i,j) = \Pi(e_i(m_1)e_j^{\top}(m_2))$ for
any $1\leq i \leq m_1$, $1\leq j \leq m_2$.
\begin{proof}
Clearly, we have $\|X\| = 1$. Furthermore, under Assumption
\ref{assumkappa}, we have
\begin{equation}\label{eq:norms_of_x}
\max\left\lbrace \|\mathbb E (X)\|,\|\mathbb E
(X)^{\top}\|\right\rbrace \leq \sqrt{\frac{c_1'}{m_1m_2}}, \quad
\frac{c_1}{m_1\wedge m_2}\leq \sigma_X^2 \leq \frac{c_1'}{m_1\wedge
m_2}.
\end{equation}
Indeed, Assumption \ref{assumkappa} implies that
\begin{equation}\label{intermLem1}
0<\frac{c_1}{m_1m_2} \leq \pi(i,j)\leq \frac{c_1'}{m_1m_2},\quad
\forall i,j.
\end{equation}
Next, we have
\begin{align*}
\|\mathbb E (X)\| &= \max_{x\in
\R^{m_2}:|x|_2=1}\sqrt{\sum_{i}\left(\sum_{j} \pi(i,j)x_j\right)^2}.
\end{align*}
Note that the maximum is clearly achieved at point $x$ satisfying
$x_j\geq 0$ for any $1\leq j \leq m_2$ since $\pi(i,j)> 0$ for any
$i,j$ in view of the two above displays. Thus, we get
\begin{align*}
\|\mathbb E (X)\| & \leq \sqrt{\frac{c_1'}{m_1m_2}}\max_{x\in
\R^{m_2}:|x|_2=1}\sqrt{\sum_i\left(\sum_j
\sqrt{\pi(i,j)}x_j\right)^2}\\
&\leq \sqrt{\frac{c_1'}{m_1m_2}}\max_{x\in
\R^{m_2}:|x|_2=1}\sqrt{\sum_i \left(\sum_j
\pi(i,j)\right)\left (\sum_j x_j^2\right)}\\
&\leq \sqrt{\frac{c_1'}{m_1m_2}}\sqrt{\sum_i \left(\sum_j
\pi(i,j)\right)}\leq \sqrt{\frac{c_1'}{m_1m_2}},
\end{align*}
where we have used successively Cauchy-Schwarz's inequality,
$|x|_2=1$ and $\sum_{i,j}\pi(i,j)=1$. Similarly, We obtain the same
bound for $\|\mathbb E\big (X\big)^{\top}\|$.

We have
$$
\sigma_X^2=\max\left\lbrace  \max_{1\leq i \leq m_1}\left(
\sum_{j=1}^{m_2} \pi(i,j) \right) , \max_{1\leq j \leq
m_2}\left(\sum_{i=1}^{m_1}\pi(i,j) \right)\right\rbrace.
$$
Combining the above display with (\ref{intermLem1}) yields the
second part of (\ref{eq:norms_of_x}).

Next, observe that for $\tilde X = X - \mathbb E( X)$, we have in
view of (\ref{eq:norms_of_x}) and the technical condition $2c_1'\leq
M c_1$ that
\begin{eqnarray}\label{eq:norms_of_x_1}
&& \frac {c_1}{2m_1\wedge m_2} \leq \sigma_{\tilde X}^2\leq \frac
{2c_1'}{m_1\wedge m_2}.
\end{eqnarray}
Indeed, this follows from the easy fact
$$
 \|\mathbb E\big( X X^\top\big)\| -\|\mathbb E\big( X\big)\|\| \mathbb E\big( X\big)^\top\|\leq \|\mathbb E
\big(\tilde X \tilde X^\top\big)\| \leq \|\mathbb E \big(X
X^\top\big)\| +\|\mathbb E\big( X\big)\| \| \mathbb
E\big(X\big)^\top\|,
$$
combined with (\ref{eq:norms_of_x}) and Assumption \ref{assumkappa}.
We proceed similarly for $\|\mathbb E \tilde X^\top \tilde X\|$.

Now,
  \begin{eqnarray}\label{interm1}
    \left\| {\bf M}_1\right\|
    &\leq& \left\| \frac{1}{n}\sum_{i=1}^n \xi_i (X_i-\E X_i)\right\| + \left\| \frac{1}{n}\sum_{i=1}^n \xi_i \E(X_i) \right\|\nonumber\\
    &\leq& \left\| \frac{1}{n}\sum_{i=1}^n \xi_i \left(X_i-\E X\right)\right\| + \sqrt{\frac{c_1'}{m_1m_2}}\left| \frac{1}{n}\sum_{i=1}^n \xi_i
    \right|.
  \end{eqnarray}
Set $Z_i=\xi_i \left(X_i-\E X\right)$. These are i.i.d. random
matrices having the same distribution as a random matrix $Z$. Since
$\|X\|=1$ we have that $\|Z_i\|\le 2|\xi_i|$, and thus Assumption
\ref{sup-exp assum} implies that $U_Z^{(\beta)}\leq c \sigma$ for
some constant $c>0$. Furthermore, in view of
(\ref{eq:norms_of_x_1}), we have $\sigma_Z\le c_2\sigma/(m_1\wedge
m_2)^{1/2}$ for some constant $c_2>0$ depending only on $c_1'$ and
$\sigma_Z\ge \bar c_3\sigma/(m_1\wedge m_2)^{1/2}$ for some constant
$c_3>0$ depending only on $c_1,\bar c$. Using these remarks we can
deduce from Proposition \ref{prop:Bernstein_unbounded} that there
exists an absolute constant $\tilde{C}>0$ such that for any $t>0$
with probability at least $1-e^{-t}$ we have
\begin{align*}
   &\left\| \frac{1}{n}\sum_{i=1}^n \xi_i
\left(X_i-\E X\right)\right\|\\
&\hspace{0.5cm}\leq \tilde{C}\max\left\{\sigma_{Z}
\sqrt{\frac{t+\log (m)}{n}}\,, \ \ U_Z^{(\beta)}\left(\log
    \frac{U_Z^{(\beta)}}{\sigma_Z}\right)^{1/\beta}\frac{t +
    \log (m)}{n}\right\}
    \\
    &\hspace{0.5cm}\leq C\sigma\max\left\{ \sqrt{\frac{t+\log(m)}{(m_1\wedge
m_2)n}}\, , \ \frac{(t+\log(m))\log^{1/\beta}(m_1\wedge m_2)}{n}
\right\}\,.
\end{align*}
Finally, in view of Assumption (\ref{sup-exp assum}) and Bernstein's
inequality for sub-exponential noise, we have for any $t>0$, with
probability at least $1-e^{-t}$,
\begin{eqnarray*}
\left| \frac{1}{n}\sum_{i=1}^n\xi_i\right| &\leq& C\sigma
\max\left\{ \sqrt{\frac{t}{n}} , \frac{t}{n}\right\},
\end{eqnarray*}
where $C>0$ depends only on $\tilde c$. We complete the proof by
using the union bound.
\end{proof}

We now treat $\|{\bf M}_2\|$.
\begin{lemma}\label{lem:tau2}
Let $X,X_1,\ldots,X_n$ be i.i.d. random variables with distribution
$\Pi$ on $\mathcal X$ satisfying Assumption \ref{assumkappa}.
Assume, in addition, that $\max_{i,j}|a_0(i,j)|\leq a$ for some
$a>0$. Then, for all $t>0,$ with probability at least $1
  -e^{-t}$ we have
\begin{align}\label{eq:lem:tau2}
\|{\bf M}_2\|\leq 2 c_1' a \, \max \left\{
\,\sqrt{\frac{t+\log(m)}{(m_1\wedge m_2)n}}, \
\frac{2(t+\log(m))}{n} \ \right\}.
\end{align}
\end{lemma}
\begin{proof} We apply Proposition \ref{prop:Bernstein_bounded} for
the random variables $Z_i=\mathrm{tr}(A_0^\top X_i) X_i - \mathbb E
(\mathrm{tr}(A_0^\top X)X)$. Using (\ref{eq:norms_of_x}) we get
$\|Z_i\|\le 2\max_{i,j}|a_0(i,j)|$ and
$$
\sigma^2_Z \le \max \left\{ \, \|\E\big(\langle A_0,X  \rangle^2
XX^\top\big)\|,\, \|\E\big(\langle A_0,X  \rangle^2 X^\top
X\,\big)\| \,\right\}\le a^2\frac{c_1'}{m_1\wedge m_2}\,.
$$
Thus, (\ref{eq:lem:tau2}) follows from Proposition
\ref{prop:Bernstein_bounded}.
\end{proof}

\subsection{Proof of Theorem \ref{th:lower_completion}}
\begin{proof}
We assume w.l.o.g. that $M=m_1\vee m_2=m_1\geq m_2$. The idea is to
adapt to our context Theorem 5 in \cite{KLT}. Note that Theorem 5 is
established under a restricted isometry condition in expectation
(See Assumption 2 in \cite{KLT}). A quick investigation of the proof
shows that the conclusion of this theorem is still valid for
$X_1,\ldots,X_n$ i.i.d. with distribution $\Pi$ satisfying
Assumption \ref{assumkappa}. Indeed, we then have for any $A\in
\R^{m_1\times m_2}$
\begin{equation}\label{interm3}
 \frac{c_1}{m_1m_2} \|A\|_2^2\leq  \|A\|_{L_2(\Pi)}^2 \leq
\frac{c_1'}{m_1m_2}\|A\|_2^2,
\end{equation}
Recall that \cite{KLT} established in the proof of Theorem 5 the
existence of a subset $\AA^0\subset\mathcal{A}(r,a)$ with
cardinality $\mathrm{Card}(\AA^0) \geq 2^{rm_1/8}+1$ containing the
zero $m_1\times m_2$ matrix ${\bf 0}$ and such that, for any two
distinct elements $A_1$ and $A_2$ of $\AA^0$,
\begin{equation}\label{lower_2}
\frac{\gamma^2}{16}(\sigma \wedge a)^2\frac{m_1^2m_2r}{n}\leq
\Arrowvert A_1-A_2\Arrowvert_{2}^2 \leq \gamma^2(\sigma \wedge
a)^2\frac{m_1^2m_2r}{n}.
\end{equation}

Next, using (\ref{interm3}) instead of Assumption 2 in \cite{KLT},
Equations (4.3) and (4.4) in \cite{KLT} are replaced respectively by
\begin{align}\label{eq: condition A}
\Arrowvert A_1-A_2\Arrowvert_{L_2(\Pi)}^2\geq &
c_1\frac{\gamma^2}{16}(\sigma \wedge a)^2\frac{m_1r}{n},
\end{align}
and
\begin{equation}\label{KLdiv}
K\big(\P_{{\bf 0}},\P_{A}\big)\ =\
\frac{n}{2\sigma^2}\|A\|_{L_2(\Pi)}^2 \leq
c_1'\frac{\gamma^2}{2}m_1r,
\end{equation}
where $K\big(\P_{{\bf 0}},\P_{A}\big)$ is the Kullback-Leibler
distance between $\P_{{\bf 0}}$ and $\P_{A}$ and $\gamma>0$ is some
numerical quantity introduced in the construction of the set $\AA^0$
in \cite{KLT}.

For any two distinct matrices $A_1,A_2$ of $\AA^0$, we have
\begin{equation}\label{eq:lower_spectral_1}
  \|A_1 - A_2\|_\infty\geq \sqrt{\frac{c_1}{c_1'}}\sqrt{\frac{\gamma}{16}}(\sigma\wedge a)\sqrt{\frac{m_1^2m_2}{n}}.
\end{equation}
Indeed, if (\ref{eq:lower_spectral_1}) does not hold, we get
$$
\|A_1 - A_2\|_{L_2(\Pi)}^2 \leq
\frac{c_1}{m_1m_2}\mathrm{rank}(A_1-A_2) \|A_1 - A_2\|_\infty^2
<c_1\frac{\gamma}{16}(\sigma\wedge
  a)^2 \frac{m_1r}{n},
$$
since $\mathrm{rank}(A_1-A_2)\leq r$ by construction of $\AA^0$ in
\cite{KLT}. This contradicts (\ref{eq: condition A}).

We now take $\gamma>0$ sufficiently small depending only on
$c_1',c_1,\alpha$ with $\alpha>0$ so that
\begin{equation}\label{eq: condition C}
\frac{1}{\mathrm{Card}(\AA^0)-1} \sum_{A\in\AA^0}K(\P_{\bf
0},\P_{A})\ \leq\ \alpha \log \big(\mathrm{Card}(\AA^0)-1\big).
\end{equation}
Combining (\ref{eq:lower_spectral_1}) with (\ref{eq: condition C})
and Theorem 2.5 in \cite{tsy_09} gives the result.
\end{proof}

\end{document}